\documentclass[a4paper,reqno]{amsart}
\usepackage{amssymb, amsmath, amscd}

\def\mapright#1{\smash{\mathop{\longrightarrow}\limits\sp{#1}}}
\newtheorem{theorem}{Theorem}[section]
\newtheorem{lemma}[theorem]{Lemma}
\newtheorem{remark}[theorem]{Remark}
\newtheorem{definition}[theorem]{Definition}

\newtheorem{example}[theorem]{Example}
\makeatletter
 \@addtoreset{equation}{section}
\makeatother
\begin{document}
\title{Homothety Curvature Homogeneity and Homothety Homogeneity}
\author{E. Garc\'{\i}a-R\'{\i}o, P.  Gilkey  \text{and} S. Nik\v cevi\'c}
\address{EG: Faculty of Mathematics,
University of Santiago de Compostela, Spain}
\email{eduardo.garcia.rio@usc.es}
\address{PG: Mathematics Department, \; University of Oregon, \;\;
  Eugene \; OR 97403 \; USA}
\email{gilkey@uoregon.edu}
\address{SN: Mathematical Institute, Sanu, Knez Mihailova 36, p.p. 367,
11001 Belgrade, Serbia}
\email{stanan@mi.sanu.ac.rs}

\begin{abstract}
We examine the difference between several notions of curvature homogeneity
and show that the notions introduced by Kowalski and Van\v{z}urov\'{a} are
genuine generalizations of the ordinary notion
of $k$-curvature homogeneity.
The homothety group plays an essential role in the analysis. We give a complete
classification of homothety homogeneous manifolds which are not homogeneous
and which are not VSI and show that such manifolds are cohomogeneity one. We also give a
complete description of the local geometry
if the homothety character defines a split extension.
\end{abstract}
\maketitle

\section{Introduction}
Let $\mathcal{M}=(M,g)$ be a pseudo-Riemannian manifold of
dimension $m\ge2$.
Let $\nabla^kR\in\otimes^{k+4}T^*M$ denote
the $k^{\operatorname{th}}$ covariant derivative of the curvature tensor
and let  $\nabla^k\mathfrak{R}\in
\otimes^{k+2}T^*M\otimes\operatorname{End}(TM)$
denote the $k^{\operatorname{th}}$ covariant derivative
of the curvature operator. These are related by the identity:
$$
\nabla^kR(x_1,x_2,x_3,x_4;x_5,\dots,x_{k+4})
=g(\nabla^k\mathfrak{R}(x_1,x_2;x_5,\dots,x_{k+4})x_3,x_4)\,.
$$
\begin{definition}\label{defn-1.1}\rm
$\mathcal{M}$ is said to be {\it k-curvature homogeneous}
if given $P,Q\in M$,  there is a linear isometry
$\phi:T_PM\rightarrow T_QM$ so that
$\phi^*(\nabla^\ell R_Q)=\nabla^\ell R_P$ for $0\le\ell\le k$.
\end{definition}

There is a slightly different version of curvature homogeneity that
we shall discuss here and which,
motivated by the seminal work of
Kowalski and Van\v{z}urov\'{a} \cite{KV11,KV12}, we shall call
{\it homothety $k$-curvature homogeneity}.
In Definition~\ref{defn-1.1}, we may replace the curvature tensor $R$
by the curvature operator $\mathfrak{R}$ since we are
dealing with isometries. This is not the case
when we deal with homotheties and the variance is crucial.
We will establish the following result in Section~\ref{sect-2}:
\begin{lemma}\label{lem-1.2}
The following conditions
are equivalent and if any is satisfied, then $\mathcal{M}$ will be said to be
{\bf homothety $k$-curvature homogeneous}:
\begin{enumerate}
\item Given any two points $P,Q\in M$, there is a linear homothety
$\Phi=\Phi_{P,Q}$ from  $T_PM$ to $T_QM$  so that if $0\le\ell\le k$, then
$\Phi^*(\nabla^\ell\mathfrak{R}_Q)=\nabla^\ell\mathfrak{R}_P$.
\item Given any two points $P,Q\in M$, there exists a linear isometry
$\phi=\phi_{P,Q}$ from $T_PM$ to $T_QM$
and there exists $0\ne\lambda=\lambda_{P,Q}\in\mathbb{R}$
so that if $0\le\ell\le k$, then
$\phi^*(\nabla^\ell R_Q)=\lambda^{-\ell-2}\nabla^\ell R_P$.
\item There exist constants $\varepsilon_{ij}$ and constants
 $c_{i_1\dots i_{\ell+4}}$ such that
 for all $Q\in M$, there is a basis $\{\xi_1^Q,\dots,\xi_m^Q\}$ for $T_QM$ and there exists
 $0\ne\lambda_Q\in\mathbb{R}$ so that if
 $0\le\ell\le k$, then for all $i_1,i_2,\dots$,
 $g_Q(\xi^Q_{i_1},\xi^Q_{i_2})=\varepsilon_{i_1i_2}$ and
 $$
\nabla^\ell R_Q(\xi^Q_{i_1},\xi^Q_{i_2},\xi^Q_{i_3},\xi^Q_{i_4};\xi^Q_{i_5},\dots,\xi^Q_{\ell+4})
 =\lambda_Q^{-\ell-2}c_{i_1\dots i_{\ell+4}}\,.
$$
 \end{enumerate}
 \end{lemma}

\begin{remark}\rm
This agrees with Proposition~0.1 of Kowalski and Van\v{z}urov\'{a} \cite{KV12}.
If we can take $\lambda_{P,Q}=1$ for all $P$ and $Q$, then
$\mathcal{M}$ is $k$-curvature homogeneous.
But we shall see in Theorem~\ref{T1.14},
 there are examples which are homothety $2$-curvature homogeneous which
 are not $2$-curvature homogeneous and thus $\lambda$ varies with the point.
\end{remark}

Motivated by Lemma~\ref{lem-1.2}, we make the following:
\begin{definition}\rm
A $k$-curvature model is a tuple
$\mathfrak{M}_k:=(V,\langle\cdot,\cdot\rangle, A^0,\dots,A^k)$
where $\langle\cdot,\cdot\rangle$
is a non-degenerate inner product on an $m$-dimensional real
vector space $V$ and where
$A^i\in\otimes^{4+i}(V^*)$ satisfies the algebraic identities of $\nabla^i R$. 
We say that two
$k$-curvature models $\mathfrak{M}_k^1$ and $\mathfrak{M}_k^2$ are
{\it homothety isomorphic} if there
is a linear isometry $\phi$ from $(V^1,\langle\cdot,\cdot\rangle^1)$ to
 $(V^2,\langle\cdot,\cdot\rangle^2)$
and if there exists $\lambda\in\mathbb{R}$ so that
$$
\phi^*A^{\ell,2}=\lambda^{-\ell-2}A^{\ell,1}\text{ for all }0\le\ell\le k\,.
$$
\end{definition}

Lemma~\ref{lem-1.2} shows
a pseudo-Riemannian manifold $\mathcal{M}$ is
homothety $k$-curvature homogeneous if and only if
there exists a $k$-curvature model $\mathfrak{M}_k$ so that
$\mathfrak{M}_k$  is homothety isomorphic to
$(T_PM,g_P,R_P,\dots,\nabla^kR_P)$ for all $P$ in $M$.

\subsection{Structure Groups} Let $\mathcal{D}(M)$ denote the group of diffeomorphisms of a pseudo-Riemannian manifold $\mathcal{M}$. We
define the group of isometries $\mathcal{I}(\mathcal{M})$ and the group of homotheties $\mathcal{H}(\mathcal{M})$
by setting:
\begin{eqnarray*}
&&\mathcal{I}(\mathcal{M}):=\{T\in\mathcal{D}(M):T^*g=g\},\\
&&\mathcal{H}(\mathcal{M}):=\{T\in\mathcal{D}(M):
\exists\lambda(T)>0:T^*g=\lambda^2g\}\,.
\end{eqnarray*}
We say that $\mathcal{M}$ is {\it homogeneous} if
$\mathcal{I}(\mathcal{M})$ acts transitively on $M$. Similarly,
$\mathcal{M}$ is said to be {\it homothety homogeneous} if
$\mathcal{H}(\mathcal{M})$ acts transitively
on $M$. There are similar local notions where the transformation $T$
is not assumed globally defined.

A complete Riemannian manifold which admits a non-Killing homothety vector field must be flat 
\cite{Tashiro}, and hence it follows that a non-flat complete homothety homogeneous manifold is necessarily 
homogeneous in the Riemannian setting; we give another proof of this result based on the slices
of Theorem~\ref{T1.6}. The situation is not so rigid in the Lorentzian case where 
pp-wave metrics support non-Killing
homothety vector fields (see for example \cite{Alek, Ku-Ra, Steller} and references therein).

\subsection{Stability}
Assertion~(1) in
the following result was established by Singer \cite{S60} in the
Riemannian context and by Podesta and Spiro \cite{PS96}
in the pseudo-Riemannian setting. In Section~\ref{sect-3}, we
will use results of \cite{PS96} to establish Assertion~(2) which
extends these results to the homothety setting. Recall that the linear orthogonal
group $\mathcal{O}$ and the linear homothety group $\mathcal{HO}$
in dimension $m$ but arbitrary signature satisfy:
$$
\dim\{\mathcal{O}\}:=\textstyle\frac12m(m-1)\text{ and }
\dim\{\mathcal{ HO}\}=\textstyle\frac12m(m-1)+1\,.$$

\begin{theorem}\label{thm-1.5}
Let $\mathcal{M}=(M,g)$ be a pseudo-Riemannian manifold.
\begin{enumerate}
\item The following Assertions are equivalent:
\begin{enumerate}
\item $\mathcal{M}$ is locally homogeneous.
\item $\mathcal{M}$ is $k$-curvature homogeneous for all $k$.
\item $\mathcal{M}$ is $k$-curvature homogeneous for $k=\frac12m(m-1)$.
\end{enumerate}
\item The following Assertions are equivalent:
\begin{enumerate}
\item $\mathcal{M}$ is locally homothety homogeneous.
\item $\mathcal{M}$ is homothety $k$-curvature homogeneous for all $k$.
\item $\mathcal{M}$ is homothety $k$-curvature homogeneous for
$k=\frac12m(m-1)+1$.
\end{enumerate}
\end{enumerate}
\end{theorem}

\subsection{Homothety homogeneous manifolds}

Suppose that $\mathcal{M}=(M,g)$ is a connected homothety homogeneous
pseudo-Riemannian manifold. 
Let $\mathcal{H}$  be a connected Lie subgroup of the group of 
homotheties of $\mathcal{M}$ which acts transitively on $\mathcal{M}$; 
note that $\mathcal{H}$ is not necessarily the full group of homotheties. If $\Phi\in\mathcal{H}$, let 
$\lambda(\Phi)>0$ be the homothety constant so that $\Phi^*(g)={\lambda(\Phi)^2}g$. Then $\lambda$
is a {\it smooth multiplicative character} on $\mathcal{H}$ since
$\lambda(\Phi_1\Phi_2)=\lambda(\Phi_1)\lambda(\Phi_2)$ for $\Phi_i\in\mathcal{H}$. 
Let $\mathcal{I}=\ker(\lambda)$
be the subgroup consisting of all the isometries which are contained in $\mathcal{H}$. 
We suppose $\mathcal{I}\ne\mathcal{H}$ or, equivalently, that $\lambda$ is non-trivial.
Let  $\mathcal{R}$ be
a Weyl scalar invariant of order $\ell$  
(i.e., it involves a total of $\ell$ derivatives of the metric tensor); 
such invariants are constructed from
the covariant derivatives of the curvature tensor by contracting indices in pairs where
we adopt the {\it Einstein convention} and sum over repeated indices. For example,
the scalar curvature $\tau$, the norm of the curvature tensor $|R|^2$, the norm of the Ricci
tensor $|\rho|^2$, and the Laplacian of the scalar curvature $\Delta\tau$ are Weyl scalar
invariants of orders $2$, $4$, $4$, and $4$, respectively, that are defined by:
$$\begin{array}{ll}
\tau:=g^{il}g^{jk}R_{ijkl},&|R|^2:=g^{ia}g^{jb}g^{kc}g^{ld}R_{ijkl}R_{abcd},\\
|\rho|^2:=g^{ia}g^{jb}g^{kc}g^{ld}R_{ijbl}R_{akcd},&\Delta\tau:=-g^{ia}g^{jb}g^{kc}R_{ijba;kc}\,.
\end{array}$$
If $\mathcal{R}$ is a Weyl scalar invariant of order $\ell$,
then the computations used to establish Lemma~\ref{lem-1.2} show that:
\begin{equation}\label{E1.a}
\Phi^*(\mathcal{R})=\lambda(\Phi)^{-\ell}\mathcal{R}\quad\text{ for any }\quad\Phi\in\mathcal{H}\,.
\end{equation}
The manifold $\mathcal{M}$
is said to be a VSI (vanishing scalar invariants) manifold if
all the Weyl scalar invariants vanish identically;
we refer to \cite{AMCH12,CMPP04, CHP08,CHP10} for further
details. In the Riemannian setting, $\mathcal{M}$ is VSI if and only if $\mathcal{M}$ is flat. There
are non-flat VSI manifolds in the higher signature setting which are homothety homogeneous
as we shall discuss in Theorem~\ref{T1.14}. If $\mathcal{M}$ is not VSI, then 
$\mathcal{M}$ is not homogeneous since some Weyl scalar invariant is non-constant. 

Let
 $\mathcal{R}$ be a Weyl scalar invariant
of order $\ell$ which does not vanish identically at some point; by Equation~(\ref{E1.a}),
$\mathcal{R}$ never vanishes on $M$ since $\mathcal{H}$ acts transitively.
In the Riemannian setting, we
may take $\mathcal{R}=|R|^2$.
If $P_0$ is the base point of $M$, set:
\begin{equation}\label{E1.b}
\mu_{\mathcal{R}}(P):=\left|\frac{\mathcal{R}(P_0)}{\mathcal{R}(P)}\right|^{1/\ell}\,.
\end{equation}
We will establish the following Theorem  in Section~\ref{S4}
which shows that $\mathcal{M}$ has cohomogeneity $1$:

\begin{theorem}\label{T1.6}
Let $\mathcal{M}=(M,g)$ be a connected pseudo-Riemannian manifold which is homothety
homogeneous with non-trivial homothety character $\lambda$ which is not VSI. 
Let $\mathcal{R}$ be a scalar Weyl invariant which does not vanish on $M$. Then
\begin{enumerate}
\item The level sets $M_c^\mathcal{R}=\{P\in M:\mu_{\mathcal{R}}(P)=c\}$ are smooth
submanifolds of $M$ which have codimension $1$, which are independent of the particular
$\mathcal{R}$ chosen, and on which $\mathcal{I}$ acts transitively.
\item If $\mathcal{M}$ is Riemannian, then there exists a constant $\kappa=\kappa(M)>0$ so that
$\operatorname{dist}(M_c^\mathcal{R},M_d^\mathcal{R})=\kappa|c-d|$. Furthermore, $\mathcal{M}$ is incomplete.
\end{enumerate}
\end{theorem}

In order to describe the local structure of homothety homogeneous manifolds
which are not VSI, we make the following:

\begin{definition}\label{D1.7}
\rm A quadruple $\mathcal{P}=(M,g_0,\mathcal{H},\lambda)$ is said to be an 
{\it $\mathfrak{P}$-structure} if:
\begin{enumerate}
\item $\mathcal{H}$ is a connected Lie group acting smoothly, transitively, and effectively
by isometries on the connected pseudo-Riemannian manifold $\mathcal{M}_0=(M,g_0)$.
\item $\lambda$ is a non-trivial, smooth $\mathbb{R}^+$ valued multiplicative character 
of $\mathcal{H}$.
\item Let $P_0$ be the base point of $M$ and let $\mathcal{H}_0:
=\{\Phi\in\mathcal{H}:\Phi P_0=P_0\}$ be the isotropy subgroup of the action. 
Then $\mathcal{H}_0\subset\ker(\lambda)$.
\end{enumerate}\end{definition}

If $\mathcal{P}=(M,g_0,\mathcal{H},\lambda)$ is a $\mathfrak{P}$ structure, let 
$M_t:=\lambda^{-1}\{t\}/\mathcal{H}_0$.
Since  $\ker(\lambda)$ acts transitively by isometries on each $M_t$,
$\mathcal{M}_t:=(M_t,{ g_{0}}|_{M_t})$ is a homogeneous submanifold of $M$
which has codimension $1$. 
If $P\in M_t$, let $g_\lambda(P)={ t^2}g_0(P)$ define a metric which is conformally equivalent
to $g_0$. Let $\mathcal{M}_{\mathfrak{P}}:=(M,g_\lambda)$ be the associated pseudo-Riemannian
manifold. We will establish the following result in Section~\ref{S5} classifying
homothety homogeneous manifolds with non-trivial homothety character which are not VSI; the
classification is still open in the VSI context:

\begin{theorem}\label{T1.8} 
\ \begin{enumerate}
\item If $\mathcal{P}$ is a $\mathfrak{P}$ structure, then $\mathcal{M}_{\mathfrak{P}}$ is
a homothety homogeneous manifold with homothety character $\lambda$.
\item If $\mathcal{M}=(M,g)$ is a homothety homogeneous manifold with non-trivial
homothety character which is not VSI, then $M=\mathcal{M}_{\mathcal{P}}$ for some $\mathfrak{P}$ structure.
\end{enumerate}\end{theorem}

Let $\mathcal{M}=(M,g)$ be a connected homothety homogeneous
pseudo-Riemannian manifold and let $\mathcal{H}$ be a Lie subgroup 
of homotheties acting transitively on $\mathcal{M}$. We suppose that $\mathcal{M}$ is not
VSI and that the homothety multiplicative character $\lambda$ is non-trivial. 
Let $\mathcal{L}$ be the Lie derivative. Recall that 
a vector field $X$ is said to be a {\it Killing} vector field if $\mathcal{L}_Xg=0$ or, equivalently, if
$X_{a;b}+X_{b;a}=0$ where `;' represents covariant differentiation with respect
to the Levi-Civita connection. Similarly $X$ is said to be a
{\it homothety} vector field if $\mathcal{L}_Xg=\lambda g$ or, equivalently,
$X_{a;b}+X_{b;a}=2\lambda g_{ab}$. Let $\mathfrak{H}$ be the Lie algebra of $\mathcal{H}$
and let $\mathfrak{I}$ the Lie sub algebra of $\mathcal{I}$; the elements of $\mathfrak{H}$ are 
homothety vector fields and the elements of $\mathfrak{I}$ are Killing vector fields.
The short exact sequence
\begin{equation}\label{E1.c}
1\rightarrow\ker(\lambda)\rightarrow\mathcal{H}\mapright{\lambda}\mathbb{R}^+\rightarrow1
\end{equation}
 defined by the homothety character is crucial.
This sequence is said to {\it split} if $\mathcal{H}$ is isomorphic as a Lie group to 
$\mathbb{R}^+\oplus\ker(\lambda)$ and if under this isomorphism, $\lambda$ is projection
on the first factor. The sequence is split if and only if there exists $0\ne\xi\in\mathfrak{H}$ so that
$[\xi,\eta]=0$ for all $\eta\in\mathfrak{H}$; such a $\xi$ generates a $1$-parameter subgroup $\Phi_t$
such that $\lambda(\Phi_t)=t$ for $t\in\mathbb{R}^+$ that is central in $\mathcal{H}$ and gives
rise to the splitting. In the Riemannian setting, if $\mathcal{I}$ is compact (or, equivalently, a slice
$M_c$ is compact), then we can average over the group to construct such an element of 
$\mathfrak{H}$ and show the sequence is necessarily split in this setting.

In view of Theorem~\ref{T1.8}, it is perhaps worth to exhibit a Lie group with
many multiplicative $\mathbb{R}^+$ valued characters; there, are of course,
many possible such examples:
\begin{example}\rm
Let $\mathcal{H}_n$ be the group of upper triangular $n\times n$ real
matrices with positive entries on the
diagonal. A typical element $H\in\mathcal{H}$ takes the form:
$$
H:=\left(\begin{array}{llllll}
h_{11}&h_{12}&h_{13}&\dots&h_{1n}\\
0&h_{22}&h_{23}&\dots&h_{2n}\\
0&0&h_{33}&\dots&h_{3n}\\
\dots&\dots&\dots&\dots&\dots\\
0&0&0&\dots&h_{nn}
\end{array}\right)
$$
where $h_{11}>0$, $\dots$, $h_{nn}>0$, and where the $h_{ij}$ are arbitrary real 
numbers for $i<j$. If
 $0\ne\vec a=(a_1,\dots,a_n)\in\mathbb{R}^n$,
 then we may define a multiplicative $\mathbb{R}^+$
valued character $\lambda_{\vec a}$ on $\mathcal{H}$ by setting
$$\lambda_{\vec a}(H)=h_{11}^{a_1}\dots h_{nn}^{a_n}\,.$$
The associated sequence $1\rightarrow\ker(\lambda_{\vec a})\rightarrow\mathcal{H}_n\rightarrow\mathbb{R}^+\rightarrow1$
is split if and only if $a_1+\dots+a_n\ne0$. Thus there are non-split characters.
If we were to permit the $h_{ij}$ to be complex with $h_{ii}\ne0$, we could
define $\mathcal{H}_{\mathbb{C}}$ and obtain similarly
$$\lambda_{\vec a}(H)=|h_{11}|^{a_1}\dots|h_{nn}|^{a_n}\,.$$
\end{example}

\begin{definition}\label{D1.9}
\rm A quadruple $\mathcal{Q}=(N,g_N,\mathcal{I},\theta)$ is said to be a 
{\it$\mathfrak{Q}$-structure}
if $\mathcal{I}$ is a connected Lie group which acts transitively and effectively by isometries
 on a connected
pseudo-Riemannian manifold $(N,g_N)$,
and if $\theta\in C^\infty(\Lambda^1N)$ is invariant under the action of $\mathcal{I}$.
Let $t\in\mathbb{R}$, let $M=\mathbb{R}\times N$, and let
$$
\mathcal{M}_{t,\mathcal{Q}}:=(\mathbb{R}\times N,g_{t,\mathcal{Q}})\text{ where }
g_{t,\mathcal{Q}}:=e^{tx}(dx^2+dx\circ\theta+ds^2_N)\,.
$$
We say that $\mathfrak{Q}$ is a {\it Riemannian $\mathfrak{Q}$ structure}
 if $\mathcal{M}_{t,\mathcal{Q}}$ is Riemannian or, equivalently, if $(N,g_N)$ is
Riemannian and if $||\theta||_{g_N}<1$. 
\end{definition}
We will establish the following result in Section~\ref{S6}.
\begin{theorem}\label{T1.10}
\ \begin{enumerate}
\item Let $\mathcal{Q}=(N,g_N,\mathcal{I},\theta)$ be a $\mathfrak{Q}$-structure.
Then 
$g_{t,\mathcal{Q}}$ is non-degenerate if and only if $g_N(\theta,\theta)\ne1$. In this setting, 
$\mathcal{M}_{t,\mathcal{Q}}$
is homothety homogeneous with non-trivial homothety character $\lambda$ and the
short exact sequence defined by $\lambda$ splits. 
\item Let $\mathcal{M}$ be a homothety homogeneous Riemannian manifold with non-trivial
homothety character $\lambda$ which is not VSI.
If the short exact sequence defined by $\lambda$ splits and if $\xi$ is spacelike,
then $\mathcal{M}=\mathcal{M}_{t,\mathcal{Q}}$ for some $\mathfrak{Q}$-structure and
for some $t$ where $g_N(\theta,\theta)\ne1$.
\end{enumerate}
\end{theorem}

\begin{theorem}\label{T1.12xx}
If $\mathcal{Q}=(N,g_N,\mathcal{I},\theta)$ is a Riemannian $\mathfrak{Q}$-structure, then
$\mathcal{M}_{t,\mathcal{Q}}$ is 
isometric to $\mathcal{M}_{\tilde t,\tilde{\mathcal{Q}}}$ where
$\tilde{\mathcal{Q}}=(\tilde N,\tilde g_N,\tilde{\mathcal{I}},0)$ has associated $1$-form $\tilde\theta=0$
if and only if $\theta_{a;b}+\theta_{b;a}=0$ for all $(a,b)$.
\end{theorem}

Of particular interest is the case when $\theta=0$. We shall prove the following result in
Section~\ref{S7}.

\begin{theorem}\label{T1.11}
Let $\mathcal{Q}=(N,g_N,\mathcal{I},\theta)$ be a $\mathfrak{Q}$-structure with $\theta=0$. 
If 
$t\ne0$, and if $\tau_N-\frac{(m-1)(m-2)}4t^2\ne0$, then 
$\mathcal{M}$ is not flat. Such manifolds provide examples of manifolds which are $k$-homothety 
curvature homogeneous
for all $k$ but which is not $0$-curvature homogeneous.\end{theorem}

\subsection{Walker Lorentzian 3 dimensional manifolds}
Section~\ref{S9} is devoted to the study of a very specific
family of examples.
Let $\mathcal{M}=(M,g_M)$ be a $3$-dimensional Lorentzian manifold
which admits a parallel null
vector field, i.e. $\mathcal{M}$ is a $3$-dimensional Walker manifold.
Such a manifold admits local
adapted coordinates $(x,y,\tilde x)$ so that the (possibly) non-zero components of the metric are given by
$$g(\partial_x,\partial_x)=-2f(x,y),\quad g(\partial_x,\partial_{\tilde x})=g(\partial_y,\partial_y)=1\,.$$
We clear the previous notation and denote this manifold by $\mathcal{M}_f$. 

 Walker manifolds are a special class of Kundt spacetimes, and thus they are Lorentzian manifolds with constant scalar curvature invariants \cite{CHP08,CHP10}. Moreover all the scalar curvature invariants of a Walker manifold vanish and we have:

\begin{theorem}\label{T1.12} The manifolds $\mathcal{M}_f$ are all VSI manifolds.
\end{theorem}

Theorem~\ref{T1.12} shows that these form an interesting family of examples.
We have (see \cite{GNS12} Theorem 2.10 and Theorem 2.12):

\begin{theorem}\label{T1.13}
Suppose that $f_{yy}$ is never zero.
\begin{enumerate}
\item $\mathcal{M}_f$ is 1-curvature homogeneous if and only if exactly
one of the following three possibilities holds:
\begin{enumerate}
\item $f_{yy}(x,y)=ay^2$ for $a\ne0$. This manifold is symmetric.
\item $f_{yy}(x,y)=\alpha(x)e^{by}$ where $0\ne b\in\mathbb{R}$ and where $\alpha(x)$ is arbitrary.
\item $f_{yy}(x,y)=\alpha(x)$ where $\alpha(x)=c\cdot\alpha_x^{3/2}$ for some $c\ne0\in\mathbb{R}$.
\end{enumerate}
\item $\mathcal{M}$ is $2$-curvature homogeneous if and only if it falls into one of the three families, all of which are locally homogeneous:
\begin{enumerate}
\item $f=b^{-2}\alpha(x)e^{by}+\eta(x)y+\gamma(x)$ where
$0\ne b\in\mathbb{R}$, where
 $\alpha(x)\ne0$, and where
$\eta(x)=b^{-1}\alpha^{-1}(x)
\{\alpha_{xx}(x)-\alpha_x^2(x)\alpha(x)^{-1}\}$.
\item $f=a(x-x_0)^{-2}y^2+\beta(x)y+\gamma(x)$ where
$0\ne a\in\mathbb{R}$.
\item $f=ay^2+\beta(x)y+\gamma(x)$ where $0\ne a\in\mathbb{R}$.
\end{enumerate}
\end{enumerate}\end{theorem}

We will establish the following analogue of Theorem~\ref{T1.13}
for homothety curvature homogeneity in Section~\ref{S9}:

\begin{theorem}\label{T1.14}
Suppose $f_{yy}$ is never zero and non-constant.\begin{enumerate}
\item If $f_{yyy}$ never vanishes, then
$\mathcal{M}_f$ is homothety 1-curvature
homogeneous.
\item If $f_{yy}=\alpha(x)$ with $\alpha_x$ never zero, then $\mathcal{M}_f$
is homothety 1-curvature homogeneous if and only if
$f=a(x-x_0)^{-2}y^2+\beta(x)y+\gamma(x)$
where $0\ne a\in\mathbb{R}$. This manifold is locally homogeneous.
\item Assume that $\mathcal{M}_f$ is homothety 2-curvature homogeneous,
and that $f_{yy}$ and $f_{yyy}$ never vanish.
Then $\mathcal{M}_f$ is locally isometric to one of the examples:
\begin{enumerate}
\item $\mathcal{M}_{\pm e^{ay}}$ for some $a\ne0$ and
$M=\mathbb{R}^3$. This manifold is
homogeneous.
\item  $\mathcal{M}_{\pm\ln(y)}$ and $M=\mathbb{R}\times(0,\infty)\times\mathbb{R}$. This manifold is a homothety homogeneous manifold which is
not locally homogeneous but which is cohomogeneity one.
\item $\mathcal{M}_{\pm y^\varepsilon}$ for $\varepsilon\ne0,1,2$ and
$M=\mathbb{R}\times(0,\infty)\times\mathbb{R}$. This manifold is a homothety homogeneous
 manifold which is not locally homogeneous but which is cohomogeneity one.
\end{enumerate}\end{enumerate}
\end{theorem}

We also refer to recent work of Dunn and McDonald \cite{DuMc2013}
for related work on homothety curvature homogeneous manifolds.

\subsection{Variable homothety curvature homogeneity}
In fact, the definition we have used in this paper differs subtly but in an
important fashion from that originally given by Kowalski and Van\v{z}urov\'{a} \cite{KV12};
in that paper the scaling constant $\lambda$ was permitted to depend
on $\ell$ and this gives rise to the notion of
{\it variable homothety $k$-curvature homogeneity}. There are 4 different
definitions which may be summarized as follows; we repeat two of
the definitions to put the new definitions in context:
\begin{definition}\label{D1.15}
\rm
Let $\mathfrak{M}_k:=\{V,\langle\cdot,\cdot\rangle,{ A^0,\dots,A^k}\}$ be
a $k$-curvature model and let $\mathcal{M}=(M,g)$ be a pseudo-Riemannian manifold.
\begin{enumerate}
\item Static isometries that are independent of $k$. Recall that:
\begin{enumerate}
\item $\mathcal{M}$ is {\it $k$-curvature homogeneous with model $\mathfrak{M}_k$}
 if for any $P$ in $M$, there exists an isometry
 $\phi_P:T_PM\rightarrow V$ so
$\phi_P^*{ A^\ell}=\nabla^\ell R_P$ for $0\le\ell\le k$.
\item $\mathcal{M}$ is {\it homothety $k$-curvature homogeneous with model $\mathfrak{M}_k$}
 if for any $P$ in $M$, there is an isometry $\phi_P:T_PM\rightarrow V$
 and a scaling factor $0\ne\lambda\in\mathbb{R}$ so that
$\phi_P^*{ A^\ell}=\lambda^{2+\ell}R_P$ for $0\le\ell\le k$.
\end{enumerate}
\item Variable isometries that depend on $k$. We shall say that:
\begin{enumerate}
\item $\mathcal{M}$ is
{\it variable-$k$-curvature homogeneous with model $\mathfrak{M}_k$}
if for every $P$ in $M$ and if for every $0\le\ell\le k$,
there exist isometries $\phi_{P,\ell}:T_PM\rightarrow V$ so that
$\phi_{P,\ell}^*{ A^\ell}=\nabla^\ell R_P$.
\item We say that $\mathcal{M}$ is
{\it variable homothety $k$-curvature homogeneous with model
$\mathfrak{M}_k$} if for every$P$ in $M$ and if for every $0\le\ell\le k$,
there are linear isometries $\phi_{P,\ell}:T_PM\rightarrow V$ and scaling factors
$0\ne\lambda_\ell\in\mathbb{R}$ so that
$\phi_{P,\ell}^*{ A^\ell}=\lambda_\ell^{2+\ell}R_P$.
\end{enumerate}
\end{enumerate}\end{definition}
In Section~\ref{sect-7}, we will use the examples which were studied
in Section~\ref{S9} to show that Theorem~\ref{thm-1.5} fails in the
context of variable curvature homogeneity and hence also for variable
homothety curvature homogeneity:
\begin{theorem}\label{T1.16}
\ \begin{enumerate}
\item Let $k$ be a positive integer.
There exists $f_k$ so that $\mathcal{M}_{f_k}$ is
variable $k$-curvature homogeneous but not variable $k+1$-curvature homogeneous.
\item The manifold $\mathcal{M}_{\frac12e^xy^2}$ is
variable $k$-curvature homogeneous for all $k$, but is not homothety 1-curvature homogeneous and hence not locally homogeneous.
\item In Definition~\ref{D1.15} we have the following implications:
\smallbreak\centerline{
$(1a)\quad\Rightarrow\quad(1b)\quad\Rightarrow\quad(2b)\quad\text{and}\quad
(1a)\quad\Rightarrow\quad(2a)\quad\Rightarrow\quad(2b)$\,.}
\smallbreak\noindent
All other possible implications are false.
\end{enumerate}
\end{theorem}

\section{The proof of Lemma~\ref{lem-1.2}}\label{sect-2}

Assume that Assertion~(1) of Lemma~\ref{lem-1.2} holds.
This means that given any two points $P$ and $Q$ in $M$,
there exists a linear homothety
$\Phi:T_PM\rightarrow T_QM$ so that if $0\le\ell\le k$
and if $\{x_i\}$ are vectors in $T_PM$, then we have that:
\begin{eqnarray*}
&&g_Q(\Phi x_1,\Phi x_2)=\lambda^2g_P(x_1,x_2),\text{ and}\\
&&\Phi\left\{\nabla^\ell\mathfrak{R}_P(x_1,x_2;x_5,{ \dots,}x_{\ell+4})x_3\right\}
=\nabla^\ell\mathfrak{R}_Q(\Phi x_1,\Phi x_2;\Phi x_5,\dots,\Phi x_{\ell+4})\Phi x_3\,.
\end{eqnarray*}
Taking the inner product with $\Phi x_4$ permits us to rewrite the
second condition, which involves the curvature operator,
in terms of the curvature tensor:
\begin{eqnarray*}
&&\lambda^2\nabla^\ell R_P(x_1,x_2,x_3,x_4;x_5,\dots,x_{\ell+4})\\
&=&\lambda^2g_P\left(\nabla^\ell\mathfrak{R}_P(x_1,x_2;x_5,\dots,x_{\ell+4})x_3,x_4\right)\\
&=&g_Q(\Phi\nabla^\ell\mathfrak{R}_P(x_1,x_2;x_5,\dots,x_{\ell+4})x_3,\Phi x_4)\\
&=&g_Q(\nabla^\ell\mathfrak{R}_Q(\Phi x_1,\Phi x_2;\Phi x_5,\dots,\Phi x_{\ell+4})\Phi x_3,\Phi x_4)\\
&=&\nabla^\ell R_Q(\Phi x_1,\Phi x_2,\Phi x_3,\Phi x_4;\Phi x_5,\dots,\Phi x_{\ell+4})\,.
\end{eqnarray*}
We set $\phi:=\lambda^{-1}\Phi$. We can rewrite these equations in the form:
\begin{eqnarray*}
&&g_Q(\phi x_1,\phi x_2)=\lambda^{-2}g_Q(\Phi x_1,\Phi x_2)=g_P(x_1,x_2),\\
&&\lambda^2\nabla^\ell R_P(x_1,x_2,x_3,x_4;x_5,\dots,x_{\ell+4})\\
&=&\nabla^\ell R_Q(\Phi x_1,\Phi x_2,\Phi x_3,\Phi x_4;\Phi x_5,\dots,\Phi x_{\ell+4})\\
&=&\lambda^{\ell+4}\nabla^\ell R_Q(\phi x_1,\phi x_2,\phi x_3,\phi x_4;\phi x_5,\dots,\phi x_{\ell+4})\\
&=&\lambda^{\ell+4}\phi^*(\nabla^\ell R_Q)(x_1,x_2,x_3,x_4;x_5,\dots,x_{\ell+4})\,.
\end{eqnarray*}
This shows $\phi$ is an isometry from $T_PM$ to $T_QM$ so
$\phi^*(\nabla^\ell R_Q)=\lambda^{-2-\ell}\nabla^\ell R_P$. Consequently in
Lemma~\ref{lem-1.2}, Assertion~(1) $\Rightarrow$
Assertion~(2);
the proof of the converse implication is similar and will be omitted.

Suppose that Assertion~(2) of Lemma~\ref{lem-1.2} holds.
Fix a basis $\{\xi_1^{P_0},\dots,\xi_m^{P_0}\}$
for $T_{P_0}M$ where $P_0$ is the base point of $M$. Set 
$c_{i_1,\dots,i_{\ell+4}}:=\nabla^\ell R(\xi_{i_1},\dots,\xi_{i_{\ell+4}})$ and
set $\varepsilon_{ij}:=g_{P_0}(\xi_i,\xi_j)$.
Let $Q\in M$. By assumption,
there is an isometry $\phi$ from $T_{P_0}M$ to $T_QM$ so:
$$\phi^*(\nabla^\ell R_Q)=\lambda_{Q}^{-\ell-2}\nabla^\ell R_{P_0}\text{ for }0\le\ell\le k\,.$$
Set $\xi_i^Q:=\phi\xi_i^{P_0}$. Then:
\begin{eqnarray*}
&&g_Q(\xi_i^Q,\xi_j^Q)=g_Q(\phi\xi_i^P,\phi\xi_j^P)=g_{P_0}(\xi_i^{P_0},\xi_j^{P_0})=\varepsilon_{ij},\\
&&\nabla^\ell R_Q(\xi_{i_1}^Q,\dots,\xi_{i_{\ell+4}}^Q)=
\nabla^\ell R_Q(\phi\xi_{i_1}^{P_0},\dots,\phi\xi_{i_{\ell+4}}^{P_0})\\
&&\qquad=
\lambda_{Q}^{-\ell-2}\nabla^\ell R_{P_0}(\xi_{i_1}^{P_0},\dots,\xi_{i_{\ell+4}}^{P_0})=
\lambda_{Q}^{-\ell-2}c_{i_1,\dots,i_{\ell+4}}\,.
\end{eqnarray*}
This shows in Lemma~\ref{lem-1.2} that Assertion~(2)
$\Rightarrow$ Assertion~(3); the proof of the converse implication is
similar and will be omitted.
\hfill\qed

\section{The proof of Theorem~\ref{thm-1.5}}\label{sect-3}

Fix the signature $(p,q)$ throughout this section where $m=p+q$.
Let $(V,\langle\cdot,\cdot\rangle)$ be an inner product space
of signature $(p,q)$.
Let $\mathcal{O}$ and $\mathcal{HO}$
be the linear orthogonal group and linear homothety group of signature $(p,q)$
respectively;
\begin{eqnarray*}
&&\mathcal{O}:=\{T\in\operatorname{GL}(V):
T^*\langle\cdot,\cdot\rangle=\langle\cdot,\cdot\rangle\},\\
&&\mathcal{HO}:=\{T\in\operatorname{GL}(V):
T^*\langle\cdot,\cdot\rangle=\lambda^2\langle\cdot,\cdot\rangle\}
\text{ for some }0<\lambda\in\mathbb{R}\}\,.
\end{eqnarray*}
Note that $\mathcal{HO}=\mathbb{R}^+\times\mathcal{O}$.
Fix a basis $\{v_1,\dots,v_m\}$ for $V$
and let $\epsilon_{ij}:=\langle v_i,v_j\rangle$. Let $\mathcal{M}=(M,g)$
be a pseudo-Riemannian manifold of signature $(p,q)$. Let
$\mathcal{F}(M)$ be the bundle of frames $\vec u=(u_1,\dots,u_m)$ for
$TM$.
Let $\mathcal{O}(\mathcal{M})$ and $\mathcal{HO}(\mathcal{M})$
be the sub-bundles:
\begin{eqnarray*}
&&\mathcal{O}(\mathcal{M}):=\{\vec u\in\mathcal{F}(M):
g(u_i,u_j)=\epsilon_{ij}\},\\
&&\mathcal{HO}(\mathcal{M}):=\{\vec u\in\mathcal{F}(M):
g(u_i,u_j)=\lambda^\epsilon_{ij}
\text{ for some }0<\lambda\in\mathbb{R}\}\,.
\end{eqnarray*}
Then $\mathcal{O}(M)$ is a principal $\mathcal{O}$ bundle while
$\mathcal{HO}(M)$ is a principal $\mathcal{HO}$ bundle.
We use the Levi-Civita connection, which is invariant under homothety
transformations, to define a $\mathcal{HO}$ structure on $\mathcal{M}$
with a canonical connection.

Let $\mathfrak{ho}(T_P\mathcal{M})$ be the Lie algebra of the group of
homothety transformations of $T_P\mathcal{M}$. Let
$$
s_0:=\dim\{\mathfrak{ho}(T_P\mathcal{M})\}=\frac12m(m-1)+1\,.
$$
For $0\le s\le s_0$, we consider the subalgebras defined by:
\begin{eqnarray*}
&&\mathfrak{ho}^0(T_P\mathcal{M})
=\{ a\in\mathfrak{ho}(T_P\mathcal{M})\,;\,\, a\cdot R=0\},\\
&&\mathfrak{ho}^s(T_P\mathcal{M})
=\{ a\in\mathfrak{ho}^{s-1}(T_P\mathcal{M})\,;\,\, a\cdot \nabla^sR=0\}\,.
\end{eqnarray*}
Clearly
$$
\mathfrak{ho}^s(T_P\mathcal{M})\supset
\mathfrak{ho}^{s+1}(T_P\mathcal{M})\text{ for all }s$$
so we have a decreasing sequence of subalgebras of
$\mathfrak{ho}(T_P\mathcal{M})$.
Let the {\it Singer number} $s(P)$ be the first integer stabilizing this
sequence above, i.e.:
$$
\mathfrak{ho}^{s(P)+r}(T_P\mathcal{M})
=\mathfrak{ho}^{s(P)}(T_P\mathcal{M})\text{ for all }r\geq 1\,.
$$
Now the assumption that $(\mathcal{M},g)$ is
homothety $k$-curvature homogeneous for some $k\ge\frac12m(m-1)+1$
shows that the Singer number $s(P)$ is constant on $M$.
The equivalences of Theorem~\ref{thm-1.5}~(2) now
follow from the work of Podesta and Spiro \cite{PS96}.
(See also \cite{Op} for an extension to the affine setting.)

\section{The proof of Theorem~\ref{T1.6}}\label{S4}
Let $\mathcal{M}=(M,g)$ be a pseudo-Riemannian manifold which is homothety
homogeneous with non-trivial homothety multiplicative character $\lambda$ and
which is not VSI. Let $\mathcal{R}$ be a
scalar Weyl invariant which does not vanish on $M$. Let $\mu_{\mathcal{R}}$ be as
defined in Equation~(\ref{E1.b}). We use Equation~(\ref{E1.a}) to see that 
$\Phi^*\mu_{\mathcal{R}}=\lambda(\Phi)\mu_{\mathcal{R}}$ and thus
$\Phi^*d\mu_{\mathcal{R}}=\lambda(\Phi)d\mu_{\mathcal{R}}$. Since $\mathcal{H}$
acts transitively on $M$, either $d\mu_{\mathcal{R}}$ never vanishes or $d\mu_{\mathcal{R}}$
vanishes identically; this latter possibility would imply $\mu_{\mathcal{R}}$ constant and
hence $\lambda$ would be the trivial character which is false by assumption. Thus
$d\mu_{\mathcal{R}}$ is never zero and the level sets $M_c^{\mathcal{R}}$ are
smooth submanifolds of $M$ which have codimension~$1$. Furthermore,
$$
\Phi:M_c^{\mathcal{R}}\rightarrow M_{\lambda(\Phi)c}^{\mathcal{R}}\text{ for any }
\Phi\in\mathcal{H}\,.
$$
Since $\mathcal{H}$
acts transitively on $M$, $\mathcal{I}:=\ker(\lambda)$ acts transitively on $M_c^{\mathcal{R}}$.
Consequently,
$M_1^{\mathcal{R}}=\mathcal{I}\cdot P_0$ and $M_c^{\mathcal{R}}=\Phi\cdot\mathcal{I}\cdot P_0$ for any $\Phi\in\mathcal{H}$
with $\lambda(\Phi)=c$. This shows the level sets $M_c:=M_c^{\mathcal{R}}$ are in fact independent
of the choice of $\mathcal{R}$ which establishes Assertion~(1) of Theorem~\ref{T1.6}.

We begin the proof of Theorem~\ref{T1.6}~(2) by establishing the following result:

\begin{lemma}\label{L4.1}
Let $\mathcal{M}=(M,g)$ be a Riemannian manifold which is homothety
homogeneous with non-trivial homothety character $\lambda$ and which is not VSI. 
\begin{enumerate}
\item Let $c$ sufficiently close to $d$. Let $P\in M_c$.
There is a unique point $Q\in M_d$ which is the closest point to $P$ in $M_d$;
$d(P,Q)=d(M_c,M_d)$.
If $\sigma$ is the shortest unit speed geodesic from $P$ to $Q$, then
$\sigma$ is perpendicular to $M_c$ at $P$ and to $M_d$ at $Q$.
\item Let $\sigma:[0,T]\rightarrow M$ be a unit speed geodesic which is 
perpendicular to $M_{\mu_{\mathcal{R}}(\sigma(0))}$ at $\sigma(0)$.
Then $\sigma$ is perpendicular to $M_{\mu_{\mathcal{R}}(\sigma(t))}$ for any $t$ in the interval $[0,T]$.
Furthermore, $\sigma[t_0,t_1]$ is a curve which minimizes the distance 
from $M_{\mu_{\mathcal{R}}(\sigma(t_0))}$ to $M_{\mu_{\mathcal{R}}(\sigma(t_1))}$ 
for any $0\le t_0<t_1\le T$.\end{enumerate}\end{lemma}

\begin{proof}
Choose $Q_1\in M_d$ to be a point on $M_d$ which is a
closest point to $P$. There might, a-priori of course, be several such points.
Let $\sigma_1$ be the unit speed geodesic from $P$ to $Q_1$ minimizing the distance
so $\sigma_1(0)=P$ and $\sigma_1(t_1)=Q_1$. 
If $\dot\sigma(t_1)$ is not perpendicular to
$T_{Q_1}M_d$,  then we could ``cut off the leg" to construct a point $Q_2$
of $M_d$ with $d(P,Q_2)<d(P,Q_1)$. Since this would contradict the choice of $Q_1$,
we must have $\dot\sigma(t_1)\perp T_{Q_1}M_d$. 
Next suppose that $\dot\sigma(0)$ is not perpendicular
to $T_PM_c$. Then we could ``cut off the leg" to construct a point $P_1 \in M_c$
so $d(P_1,Q_1)<d(P,Q_1)$. Note that $\mathcal{I}$ acts transitively on $M_c$ for any $c$.
Choose $\phi\in\mathcal{I}$
so $\phi P_1=P$. Then:
$$d(P,\phi Q_1)=d(\phi P_1,\phi Q_1)=d(P_1,Q_1)<d(P,Q_1)$$
which again contradicts the choice of $Q_1$. This shows that the closest point is unique. Given
any other point $P_2\in M_c$, we construct $Q_2$ similarly. Choose an isometry $\phi\in\mathcal{I}$
with $\phi P_2=P$. Then we must have that $\phi Q_2=Q$ and thus
$d(P,Q)=d(P_2,Q_2)$. This proves Assertion~(1);
Assertion~(2) is an immediate consequence of Assertion~(1).\end{proof}

We now prove Theorem~\ref{T1.6}~(2).
Suppose $1<s_1<s_2$. Choose
homotheties $\Phi_i$ so $\lambda(\Phi_i)=s_i$. We then have
$$
\Phi_1M_1=M_{s_1},\quad \Phi_2M_1=M_{s_2},\quad 
(\Phi_1\Phi_2)M_1=\Phi_1M_{s_2}=\Phi_2M_{s_1}=M_{s_1s_2}
\,.$$
Therefore,
 \begin{eqnarray*}
&&d(M_1,M_{s_1s_2})=d(M_1,M_{s_1})
+d(M_{s_1},M_{s_1s_2})\\
&=&d(M_1,M_{s_1})+d(\Phi_1M_1,\Phi_1M_{s_2})
=d(M_1,M_{s_1})+\lambda(\Phi_1)d(M_1,M_{s_2})\\
&=&d(M_1,M_{s_1})+{s_1}d(M_1,M_{s_2})\,.
\end{eqnarray*}
Similarly,  we have that $d(M_1,M_{s_1s_2})=d(M_1,M_{s_2})+{s_2}d(M_1,M_{s_1})$. Thus
$$d(M_1,M_{s_1})+{s_1}d(M_1,M_{s_2})=d(M_1,M_{s_2})+{s_2}d(M_1,M_{s_1})\,.$$
Consequently $d(M_1,M_{s_1})({s_2}-1)=d(M_1,M_{s_2})({s_1}-1)$ so
$$\kappa:=\frac{d(M_1,M_{s_1})}{{s_1}-1}
=\frac{d(M_1,M_{s_2})}{{s_2}-1}$$
is independent of the choice of $s_1$ and $s_2$ for $1<s_1<s_2$.
Let $s<t$. 
Choose a homothety $\Phi$ so $\lambda(\Phi)=s$. Then $\Phi(M_{\frac ts})=M_t$. Since $1<\frac ts$ ,$$d(M_s,M_t)=d(\Phi M_1,\Phi M_{\frac ts})= sd(M_1,M_{\frac ts})
=\textstyle  s\kappa(\frac{ t}{ s}-1)=\kappa( t- s)\,.$$
Let $0<c<1$. Then $d(M_c,M_1)=\kappa(1-c)\le\kappa$.
Let $\sigma$ be a unit speed geodesic with $\sigma(0)\in M_1$, with
$\dot\sigma(0)\perp M_1$,
and with $g(\dot\sigma(0),\operatorname{grad}\mu_{ \mathcal R})=-1$. By Lemma~\ref{L4.1}
 $d(\sigma(t),M_1)=t$. Consequently $t<\kappa$ and the geodesic $\sigma$ does not
extend for infinite time. This shows $\mathcal{M}$ is incomplete which completes the proof of
Theorem~\ref{T1.6}.
\hfill\qed

\section{The proof of Theorem~\ref{T1.8}}\label{S5}

We adopt the notation of Definition~\ref{D1.7} and let
$\mathcal{P}=({ M},g_0,\mathcal{H},\lambda)$ be an $\mathfrak{P}$-structure.
 Let $P_0$ be the base point of $M$ and let
 $\mathcal{H}_0:=\{\Phi\in\mathcal{H}:{ \Phi P_0=P_0}\}$ 
 be the isotropy group. 
By assumption $\mathcal{H}_0\subset\ker(\lambda)$ so $M_1:=\ker(\lambda)/\mathcal{H}_0{ =\lambda^{-1}(\{ 1\})/\mathcal{H}_0}$ is
a smooth submanifold of $M=\mathcal{H}/\mathcal{H}_0$ of codimension $1$. 
Choose a smooth 1-parameter subgroup $\Phi_t$ of $\mathcal{H}$ (written multiplicatively) so that
$\lambda(\Phi_t)=t$ for $t\in\mathbb{R}^+$. This permits us to decompose 
\begin{equation}\label{E5.a}
\mathcal{H}=\mathbb{R}^+\times\ker(\lambda)
\end{equation}
as a differentiable manifold; the multiplicative character $\lambda$ is then projection on the
first factor. The decomposition of Equation~(\ref{E5.a}) will not, of course, in general preserve the
group structure; we shall return to this point in Section~\ref{S6}. Equation~(\ref{E5.a})
induces a corresponding decomposition
$$
{ M}=\mathcal{H}/\mathcal{H}_0=\mathbb{R}^+\times\ker(\lambda)/\mathcal{H}_0
=\mathbb{R}^+\times M_1\,.
$$
The slice $M_t:=\Phi_tM_1$ corresponds to $\{t\}\times M_1$. We have $g_\lambda(t,y)={ t^2}g(t,y)$.
We compute that
\begin{eqnarray*}
&&\{\Phi_s^*(g_\lambda)(t,y)\}(\xi_1,\xi_2)=
\{g_\lambda(st,y)\}((\Phi_s)_*\xi,(\Phi_s)_*\xi_2)\\
&=&{ s^2t^2}\{g_0(st,y)\}((\Phi_s)_*\xi,(\Phi_s)_*\xi_2)
={ s^2t^2}g_0(t,y)(\xi_1,\xi_2)={ s^2}g_\lambda(t,y)(\xi_1,\xi_2)
\end{eqnarray*}
and thus $\Phi_s^*(g_\lambda)={ s^2}g_\lambda$. On the other hand, if $\Phi\in\ker(\lambda)$, then
\begin{eqnarray*}
&&\{\Phi^*(g_\lambda)(t,y)\}(\xi_1,\xi_2)=
\{g_\lambda(t,\Phi y)\}(\Phi_*\xi_1,\Phi_*\xi_2)\\
&=&
{ t^2}g_0(t,\Phi y)(\Phi_*\xi_1,\Phi_*\xi_2)={ t^2}g_0(t,y)(\xi_1,\xi_2)=g_\lambda(t,y)(\xi_1,\xi_2)
\end{eqnarray*}
and thus $\Phi^*g_\lambda=g_\lambda$. 
This implies that $\Phi$ acts by isometries on $(M,g_\lambda)$.
Since we can decompose any $\Phi\in\mathcal{H}$ in the form 
$\Phi=\Phi_t\Psi$ for $\Psi\in\ker(\lambda)$,
we conclude $\Phi^*g_\lambda={ \lambda(\Phi)^2}g_\lambda$ for any 
$\Phi\in\mathcal{H}$. Since $\mathcal{H}$
acts transitively on $M$, this implies $(M,g_\lambda)$ is homothety homogeneous with 
homothety character $\lambda$ and $\Phi:M_t\rightarrow M_{\lambda(\Phi)t}$. 

Conversly, if $(M,g)$ is homothety homogeneous with non-trivial
homothety character $\lambda$ and if $(M,g)$ is not a VSI
manifold, then we can use Theorem~\ref{T1.6} to find the level sets $M_c$. We may
define $g_0(P):=\mu_{\mathcal{R}}(P)^{{ -2}}g(P)$ to define a conformally equivalent manifold
on which $\mathcal{H}$ acts by isometries. The corresponding $\mathfrak{P}$ structure
is then given by $(M,g_0,\mathcal{H},\lambda)$. The fact that the isotropy subgroup $H_0\subset\ker(\lambda)$ then follows; the fact that $\mathcal{M}$ is not VSI plays a crucial role in defining
the level sets.\hfill\qed

\section{The proof of Theorem~\ref{T1.10}}\label{S6}
Let $\mathcal{Q}=(N,g_N,\mathcal{I},\theta)$ be a $\mathfrak{Q}$ structure. Fix a point $P\in N$; which
point is chosen is irrelevant since $(N,g_N)$ is a homogeneous space. If $\theta=0$, then
$g_{t,\mathcal{Q}}$ is non-singular so we suppose $\theta\ne0$. First suppose that $\theta$ is
not a null covector. Choose an orthonormal basis $\{\eta_2,\dots,\eta_m\}$ for  $T_P{ N}$ so that 
${\theta(\eta_i)=0}$ for $i\ge3$, so that ${\theta(\eta_2)=c}$, and so that $g(\eta_i,\eta_j)=\epsilon_i\delta_{ij}$ where $\epsilon_i=\pm1$ for $2\le i,j\le m$. Let $\eta_1=\partial_x$. 
We then have:
$$\det(g_{ij})=e^{mtx}\det\left(\begin{array}{ll}
1&c\\c&\epsilon_2\end{array}\right)\cdot\prod_{i\ge3}\epsilon_i\,.$$
If $\epsilon_2=-1$, then $\det(g_{ij})\ne0$. In this setting $\theta$ is timelike so $g_N(\theta,\theta)<0$.
If $\epsilon_2=+1$, then $\theta$ is spacelike and the metric is non-degenerate
if and only if $c^2\ne1$
or equivalently $g_N(\theta,\theta)\ne1$. Suppose $0\ne\theta$ is a null covector. 
We can choose an
orthonormal basis $\{\eta_2,\dots,\eta_m\}$ for $T_P{ N}$ so $\eta_2$ is spacelike, so $\eta_3$ is
timelike, so $\theta(\eta_2)=\theta(\eta_3)=c$, and so $\theta(\eta_i)=0$ for $i\ge4$. We then have
$$\det(g_{ij})=e^{mtx}\det\left(\begin{array}{rrr}
1&c&c\\c&1&0\\c&0&-1\end{array}\right)\cdot\prod_{i\ge4}\epsilon_i\
=e^{mtx}(-1-c(-c)+c(-c))\ne0\,.
$$
The map $x\rightarrow e^x$ provides a
group isomorphism between $(\mathbb{R},+)$ and $(\mathbb{R}^+,\cdot)$ which we
use to write the character additively and express $\mathcal{H}=\mathbb{R}\times\ker(\lambda)$. 
If $\Phi=(a,\phi)\in\mathcal{H}$, 
let $\Phi:(x,y)\rightarrow(x+a,\phi y)$ define a transitive action of $\mathcal{H}$ on 
$M=\mathbb{R}\times N$. 
Since $\theta$ is an invariant 1-form, we have $\Phi^*g_{t,\mathcal{Q}}=e^{ta}g_{t,\mathcal{Q}}$.
Assertion~(1) now follows.

Conversly, suppose that $\mathcal{M}$ is a homothety homogeneous manifold with non-trivial
homothety character which is not VSI. Suppose the homothety character $\lambda$ splits. We normalize the homothety vector
field $\mathfrak{h}$ (which is assumed to be spacelike)
 so that $g_M(\mathfrak{h},\mathfrak{h})=1$ on $M_1$. We now
write the flow additively to construct
a diffeomorphism of $M$ with $\mathbb{R}\times M_1$. Let $\theta(\eta):=g_M(\mathfrak{h},\eta)$ for
$\eta\in TM_1$. Since $\mathfrak{h}$ is invariant under the action of $\mathcal{I}=\ker(\lambda)$,
$\theta$ is an invariant $1$-form on $M_1$ and the metric for any point of $M_1$ takes the
form $g_M=dx^2+dx\circ\theta+g_{M_1}$. Using $\Phi_x$ to push the metric from $M_1$ to
$M_{e^x1}$ then introduces the desired conformal factor for some $t$.
\hfill\qed

\section{The proof of Theorem~\ref{T1.12xx}}
Suppose that $\mathcal{M}=\mathcal{M}_{t,\mathfrak{Q}}$ is isomorphic to $\mathcal{M}_{\tilde t,\tilde{\mathcal{Q}}}$ for some $\tilde{\mathcal{Q}}$ structure with $\tilde\theta=0$. Then there exists
a homothety vector field $\tilde{\mathfrak{h}}$ on $\mathcal{M}$ which is central and which is perpendicular
to the slices. Thus we can write $\mathfrak{h}=\tilde{\mathfrak{h}}+\xi$ where $\xi$ is a killing vector field.
We then have $\theta(\eta)=g_M(\tilde{\mathfrak{h}},\eta)+g_M(\xi,\eta)=g_M(\xi,\eta)$ and thus
$\theta$ is dual to $\xi$ with respect to the metric $g_N$ on the slice $\{0\}\times N$. Consequently
$\theta_{a;b}+\theta_{b;a}=0$.

Next suppose that $\theta_{a;b}+\theta_{b;a}=0$. The associated dual vector field $\xi$ is then
a Killing vector field which is invariant under the action of $\mathcal{I}$. Since $\mathcal{I}$ acts
transitively on $N$, we can integrate $\xi$ to find a smooth $1$-parameter flow 
$\{\phi_\epsilon\}_{\epsilon\in\mathbb{R}}$ of isometries
which commutes with $\mathcal{I}$. Let $\tilde{\mathcal{I}}$  be the (possibly larger) group of isometries
generated by $\mathcal{I}$ and $\{\phi_\epsilon\}_{\epsilon\in\mathbb{R}}$.
Let $\tilde{\mathcal{Q}}:=(N,g_N,\tilde{\mathcal{I}},0)$
be the extended $\mathfrak{Q}$ structure. Let
\begin{equation}\label{EQX}
\varrho:=(1-||\xi||^2)^{-1/2}\text{ and }s:=\varrho t\,.
\end{equation}
Set
$\tilde{\mathcal{M}}_s=\mathcal{M}_{s,\tilde{\mathfrak Q}}$.
Define a diffeomorphism $\Psi$ of $M=\mathbb{R}\times N$ by setting
$$\Psi_\varrho(x,y):=(x,\phi_{\varrho x}y)\,.$$
We will show that $\Psi_\varrho$ is an isometry from $\mathcal{M}$
to $\tilde{\mathcal{M}}_s$. Fix $(x,y)\in M$. Let $\eta\in T_yN$. We compute:
$$\Psi_*\partial_x=\partial_x+\varrho\phi_{\varrho x,*}\xi\text{ and }
   \Psi_*\eta=\phi_{\varrho x,*}\eta\,.
$$
Since $\phi_{\varrho x,*}$ is an isometry, we have:
\begin{eqnarray*}
&&g_{\tilde{M}_s}
\left(\vphantom{\vrule height 9pt}\Psi_*\partial_x,\Psi_*\partial_x\right)
=e^{sx}\{1+\varrho^2||\xi||^2\},\\
&&g_{\tilde{M}_s}\left(\vphantom{\vrule height 9pt}\Psi_*\partial_x,\Psi_*\eta\right)
=e^{sx}\varrho\theta(\eta),\\
&&g_{\tilde{M}_s}(\left(\vphantom{\vrule height 9pt}\Psi_*\eta_1,\Psi_*\eta_2\right))
=e^{sx}g_N(\eta_1,\eta_2),\\
&&\Psi^*g_{s,\tilde{\mathcal Q}}=
e^{sx}((1+\varrho^2||\xi||^2)dx^2+ \varrho dx\circ\theta+g_N)\,.
\end{eqnarray*}
We use a diffeomorphism $\Upsilon$ to change variables setting $\tilde x=(1+\rho^2||\xi||^2)^{1/2}x$ and 
$\tilde t=(1+\rho^2||\xi||^2)^{-1/2}s$. Then:
$$ \Upsilon^*g_{\tilde{M}_s}=e^{\tilde t\tilde x}(d\tilde x^2+\varrho(1+\rho^2||\xi||^2)^{-1/2}
d\tilde x\circ\theta+g_N)\,.$$
We use the defining relation for $\varrho$ to see
\begin{eqnarray*}
&&\varrho(1+\varrho^2||\xi||^2)^{-1/2}=\left\{\frac1{1-||\xi||^2}\right\}^{1/2}
\left\{1+\frac{||\xi||^2}{1-||\xi||^2}\right\}^{-1/2}=1,\\
&&\tilde t=(1+\rho^2||\xi||^2)^{-1/2}s=\left(1+\frac{||\xi||^2}{1-||\xi||^2}\right)^{-1/2}s
=(1-||\xi||^2)^{1/2}s=t\,.
\end{eqnarray*}
We use the variable $\tilde x$ instead of $x$ to see that $\Psi$ provides an isometry between
$\mathcal{M}$ and $\mathcal{M}_{s,\tilde{\mathcal Q}}$.\hfill\qed

\section{The proof of Theorem~\ref{T1.11}}\label{S7}
Let $\mathcal{M}=\mathcal{M}_{t,\mathcal{Q}}$ where $\mathcal{Q}=(N,g_N,\mathcal{I},\theta)$
is a $\mathfrak{Q}$ structure with $\theta=0$.
We examine the curvature tensor. Fix $t$ and fix a point $P\in N$.
Let $g=g_N$ and let $\tilde g=g_{ t,\mathcal{Q}}$.
Choose local coordinates
$y=(y^1,\dots,y^{m-1})$ centered at $P$.
Let indices $u,v,w$ range from $0$ to $m-1$
and index the coordinate frame
$(\partial_x,\partial_{y_1},\dots,\partial_{y_{m-1}})$;
indices $i,j,k$ range from $1$ to $m-1$ and index the
coordinate frame $(\partial_{y_1},\dots,\partial_{y_{m-1}})$.
Let $\Gamma$ be the Christoffel symbols of $g$ and $\tilde\Gamma$
be the Christoffel symbols of $\tilde g$. Let $\delta_i^j$ be the Kronecker
index. We compute: 
$$\begin{array}{llll}
\tilde g_{00}=e^{tx},&\tilde g_{0i}=0,&\tilde g_{ij}=e^{tx}g_{ij},\\
\noalign{\medskip}
\tilde \Gamma_{00}{}^0=\frac12t&
\tilde \Gamma_{00}{}^i=0,\\
\tilde \Gamma_{0i}{}^0=0,&
\tilde\Gamma_{0i}{}^k=\frac12t\delta_i^k\\
\tilde \Gamma_{ij}{}^0=-\frac12tg_{ij},&
\tilde \Gamma_{ij}{}^k=\Gamma_{ij}{}^k.
\end{array}$$
Thus the covariant derivatives are given by
$$\begin{array}{lll}
\tilde\nabla_{\partial_x}\partial_x=\frac12t\partial_x,&
\tilde\nabla_{\partial_x}\partial_{y_i}=\frac12t\partial_{y_i},\\
\tilde\nabla_{\partial_{y_i}}\partial_x=\frac12t\partial_{y_i},&
\tilde\nabla_{\partial_{y_i}}\partial_{y_j}
=\Gamma_{ij}{}^k\partial_{y_k}-\frac12tg_{ij}\partial_x.
\vphantom{\vrule height 12pt}
\end{array}$$
We choose the coordinate system so the first derivatives of
$g_{ij}$ vanish at $P$ and hence $\Gamma(P)=0$.
Consequently the curvature operator at $P$ is given by:
\medbreak\quad
$\tilde{\mathcal{R}}(\partial_x,\partial_{y_i})\partial_x=
\{\tilde\nabla_{\partial_x}\tilde\nabla_{\partial_{y_i}}
-\tilde\nabla_{\partial_{y_i}}\tilde\nabla_{\partial_x}\}\partial_x
=\frac12t\tilde\nabla_{\partial_x}\partial_{y_i}
-\frac12t\tilde\nabla_{\partial_{y_i}}\partial_x=0$,
\medbreak\quad
$\tilde{\mathcal{R}}(\partial_x,\partial_{y_i})\partial_{y_j}=
\{\tilde\nabla_{\partial_x}\tilde\nabla_{\partial_{y_i}}
-\tilde\nabla_{\partial_{y_i}}\tilde\nabla_{\partial_x}\}\partial_{y_j}$
\smallbreak\qquad\qquad\qquad\quad
$=\tilde\nabla_{\partial_x}
\{\Gamma_{ij}{}^k\partial_{y_k}-\frac12tg_{ij}\partial_x\}
-\frac12t\tilde\nabla_{\partial_{y_i}}\partial_{y_j}$
\smallbreak\qquad\qquad\qquad\quad
$=\frac12t\{\Gamma_{ij}{}^k\partial_{y_k}-\frac12tg_{ij}\partial_x\}
-\frac12t\{\Gamma_{ij}{}^k\partial_{y_k}-\frac12tg_{ij}\partial_x\}=0$,
\medbreak\quad
$\tilde{\mathcal{R}}(\partial_{y_i},\partial_{y_j})\partial_x=
\{\tilde\nabla_{\partial_{y_i}}\tilde\nabla_{\partial_{y_j}}
-\tilde\nabla_{\partial_{y_j}}\tilde\nabla_{\partial_{y_i}}\}\partial_x
=\frac12t\tilde\nabla_{\partial_{y_i}}\partial_{y_j}-
\frac12t\tilde\nabla_{\partial_{y_j}}\partial_{y_i}=0$,
\medbreak\quad
$\tilde{\mathcal{R}}(\partial_{y_i},\partial_{y_j})\partial_{y_k}=
\{\tilde\nabla_{\partial_{y_i}}\tilde\nabla_{\partial_{y_j}}
-\tilde\nabla_{\partial_{y_j}}\tilde\nabla_{\partial_{y_i}}\}\partial_{y_k}$
\smallbreak\qquad\qquad\qquad\quad
$=\tilde\nabla_{\partial_{y_i}}(\Gamma_{jk}{}^\ell\partial_{y_\ell}
-\frac12tg_{jk}\partial_x)
-\tilde\nabla_{\partial_{y_j}}(\Gamma_{ik}{}^\ell\partial_{y_\ell}
-\frac12tg_{ik}\partial_x)$
\smallbreak\qquad\qquad\qquad\quad
$=R_{ijk}{}^\ell\partial_{y_\ell}-\frac14t^2g_{jk}\partial_{y_i}
+\frac14t^2g_{ik}\partial_{y_j}$.
\medbreak\noindent We can now express the scalar curvature and Ricci
tensor $\{\tilde\rho,\tilde\tau\}$ for $\tilde g$ in terms of the scalar
curvature and Ricci tensor $\{\rho,\tau\}$ for $g$ and complete the proof of Theorem~\ref{T1.11}~(4)
by computing:
\medbreak\hfill$
\tilde\rho=\rho-\textstyle\frac{m-2}4t^2g\quad \text{ and }\quad 
\tilde\tau=e^{-tx}\{\tau-\frac{(m-1)(m-2)}4t^2\}$.\hfill\vphantom{.}\qed

\section{The proof of Theorem~\ref{T1.14}}\label{S9}

Let $\mathcal{O}$ be a connected open subset of $\mathbb{R}^3$, let $f=f(x,y)\in C^\infty(\mathcal{O})$,
and let $\mathcal{M}_f=(\mathcal{O},g_f)$ where
$$
g_f(\partial_x,\partial_x)=-2f(x,y),\quad
g_f(\partial_x,\partial_{\tilde x})=g_f(\partial_y,\partial_y)=1\,.
$$
We suppress the subscript ``$f$" when no confusion is likely to ensure.
We follow the discussion in \cite{GNS12}.
The (possibly) non-zero covariant derivatives are given by:
$$\nabla_{\partial_x}\partial_x=-f_x\partial_{\tilde x}
+f_y\partial_y\text{ and }
\nabla_{\partial_x}\partial_y=\nabla_{\partial_y}\partial_x
=-f_y\partial_{\tilde x}\,.$$
This shows that 
$\operatorname{Range}(\mathcal{R})\subset\operatorname{Span}\{\partial_y,\partial_{\tilde x}\}$.
Covariantly differentiating this relationship implies similarly
$\operatorname{Range}(\nabla^k\mathcal{R})\subset\operatorname{Span}\{\partial_y,\partial_{\tilde x}\}$.
for any $k$. Furthermore, $\nabla^k\mathcal{R}(\cdot)$ vanishes if any entry is $\partial_{\tilde x}$
since $\nabla_{\tilde x}=0$ and since the metric is independent of $\tilde x$. Lowering indices
then shows that $\nabla^kR(\cdot)=0$ if any index is $\partial_{\tilde x}$. Thus the only non-zero
entries take the form $\nabla^kR(\partial_x,\partial_y,\partial_y,\partial_x;\dots)$.
Relative to the basis $\{\partial_x,\partial_{\tilde x},\partial_y\}$, the metric tensor $g_{ij}$ and
the inverse metric tensor $g^{ij}$ are given by:
$$(g_{ij})=\left(\begin{array}{rrr}-2f&1&0\\1&0&0\\0&0&1\end{array}\right)\text{ and }
(g^{ij})=\left(\begin{array}{rrr}0&1&0\\1&2f&0\\0&0&1\end{array}\right)\,.   
$$
Thus any Weyl contraction must involve a $\partial_{\tilde x}$ variable; the curvature tensor vanishes
on such variables and thus these manifolds are VSI. This establishes Theorem~\ref{T1.12}.

We now establish Theorem~\ref{T1.14}.
The (possibly) non-zero curvatures and covariant derivatives to order $2$ are:
\begin{eqnarray*}
&&R(\partial_x,\partial_y,\partial_y,\partial_x)=f_{yy},\\
&&\nabla R(\partial_x,\partial_y,\partial_y,\partial_x;\partial_x)=f_{xyy},\\
&&\nabla R(\partial_x,\partial_y,\partial_y,\partial_x;\partial_y)=f_{yyy},\\
&&{ \nabla^2} R(\partial_x,\partial_y,\partial_y,\partial_x;\partial_x,
\partial_x)=f_{xxyy}-f_yf_{yyy},\\
&&\nabla^2R(\partial_x,\partial_y,\partial_y,\partial_x;\partial_x,\partial_y)=
\nabla^2R(\partial_x,\partial_y,\partial_y,\partial_x;\partial_y,
\partial_x)=f_{xyyy},\\
&&\nabla^2R(\partial_x,\partial_y,\partial_y,\partial_x;\partial_y,
\partial_y)=f_{yyyy}.
\end{eqnarray*}

If $f_{yy}$ vanishes identically, then $\mathcal{M}$ is flat.
The vanishing of $f_{yy}$ is an invariant
of the homothety 0-model. Since we are interested in homothety curvature homogeneity,
we shall assume $f_{yy}$
never vanishes; since $M$ is connected, either $f_{yy}$ is always positive or
$f_{yy}$ is always negative.
We shall usually assume $f_{yy}>0$ as the other case is handled similarly.
The simultaneous vanishing of $f_{yyy}$ and of $f_{xyy}$ is an invariant of
the homothety 1-model. The case $f_{yy}=a$ for $0\ne a\in\mathbb{R}$
gives rise to a symmetric space. We shall therefore assume $f_{yy}$
non-constant. This gives rise to
two cases $f_{yyy}$ never zero and $f_{yyy}$ vanishing identically
but $f_{xyy}$ never zero.

\begin{definition}
\rm Let $\mathfrak{M}_{1,c_1,c_2}$ be the 1-curvature model whose (possibly) non-zero components
are defined by $ \varepsilon_{13}=\varepsilon_{22}=1$, $c_{1221}=1$, $c_{12211}=c_1$, and
$c_{12212}=c_2$:
$$\begin{array}{lll}
\langle\xi_1,{\xi_3}\rangle=1,&\langle\xi_2,\xi_2\rangle=1,&
R(\xi_1,\xi_2,\xi_2,\xi_1)=1,\\
\nabla R(\xi_1,\xi_2,\xi_2,\xi_1;\xi_1)=c_1,&
\nabla R(\xi_1,\xi_2,\xi_2,\xi_1;\xi_2)=c_2.
\end{array}$$\end{definition}

\subsection{Proof of Theorem~\ref{T1.14}~(1,2)}
We suppose $f_{yy}>0$. The distributions
$$\ker(\mathfrak{R})=\operatorname{Span}\{\partial_{\tilde x}\}\text{ and }
\operatorname{Range}(\mathfrak{R})=\operatorname{Span}\{\partial_y,\partial_{\tilde x}\}
$$
are invariantly defined. To preserve these distributions, we set:
\begin{equation}\label{E9.a}
\xi_1=a_{11}(\partial_x+f\partial_{\tilde x}+a_{12}\partial_y+a_{13}\partial_{\tilde x}),\quad
    \xi_2=\partial_y+a_{23}\partial_{\tilde x},\quad \xi_3=a_{33}\partial_{\tilde x}\,,
\end{equation}
 for some functions $a_{ij}$ on $\mathcal{O}$.
To ensure that the inner products are normalized properly, we impose the relations:
$$a_{12}^2+2a_{13}=0,\quad a_{12}+a_{23}=0,\quad a_{11}a_{33}=1\,.$$
This determines $a_{13}$, $a_{23}$, and $a_{33}$; these parameters play no further role
and $\{\lambda,a_{11},a_{12}\}$ remain as free parameters where $\lambda$
is the homothety rescaling factor.
We suppose $f_{yyy}\ne0$. Set:
\begin{equation}\label{E9.b}
\lambda:=f_{yyy}f_{yy}^{-1},\quad
a_{12}:=-f_{xyy}f_{yyy}^{-1},\quad
a_{11}^2:=f_{yy}^{-1}\lambda^2\,.
\end{equation}
We then have
\begin{eqnarray}
&&R(\xi_1,\xi_2,\xi_2,\xi_1)=a_{11}^2f_{yy}=\lambda^2,\nonumber\\
&&\nabla R(\xi_1,\xi_2,\xi_2,\xi_1;\xi_1)=a_{11}^3\{f_{xyy}+a_{12}f_{yyy}\}=0,\label{E9.c}\\
&&\nabla R(\xi_1,\xi_2,\xi_2,\xi_1;\xi_2)=a_{11}^2f_{yyy}=\lambda^2f_{yy}^{-1}f_{yyy}=\lambda^3\,.\nonumber
\end{eqnarray}
All the parameters of the theory have been determined (modulo a possible sign ambiguity in $a_{11}$) and
any homothety 1-model for $\mathcal{M}_f$ is isomorphic to $\mathfrak{M}_{1,0,1}$ in this special case. This proves Theorem~\ref{T1.14}~(1).

Suppose $f_{yy}>0$, $f_{yyy}=0$, and $f_{xyy}$ never vanishes. Set $f_{yy}=\alpha(x)$.
The parameter $a_{12}$ plays no role. To ensure that $\mathcal{M}_f$ is homothety 1-curvature
homogeneous, we impose the following relations where $\{a_{11},\lambda\}$ are unknown functions
to be determined and where $\{c_0,c_1\}$ are unknown constants:
\begin{eqnarray*}
&&R(\xi_1,\xi_2,\xi_2,\xi_1)=a_{11}^2(x)\alpha(x)=\lambda^2(x)c_0,\\
&&R(\xi_1,\xi_2,\xi_2,\xi_1;\xi_1)=a_{11}^3(x)\alpha_x(x)=\lambda^3(x)c_1,\\
&&R(\xi_1,\xi_2,\xi_2,\xi_1;\xi_2)=0\,.
\end{eqnarray*}
Consequently, $a_{11}^6(x)\alpha^3(x)=\lambda^6(x)c_0^3$ and
$a_{11}^6(x)\alpha^2_x(x)=\lambda^6(x)c_1^2$.
This shows that $\alpha^3(x)=c_3\alpha_x^2(x)$ for some constant $c_3$.
We solve this ordinary differential equation to see that
$$\alpha(x)=a(x-x_0)^{-2}\text{ for }0\ne a\in\mathbb{R}
\text{ and }x_0\in\mathbb{R}\,.$$
This has the form given in Theorem~\ref{T1.13}~(2b) and defines a locally homogeneous example. Theorem~\ref{T1.14}~(2) now follows.
\hfill\qed

\subsection{The proof of Theorem~\ref{T1.14}~(3)}
We assume that
$f_{yy}$ and $f_{yyy}$ never vanish as this case is the
only possible source of new examples not covered by
Theorem~\ref{T1.13}.
We shall suppose $f_{yy}>0$; the case $f_{yy}<0$ is handled similarly.
As any two homothety 1-curvature models for $\mathcal{M}_f$ are isomorphic,
we can adopt the normalizations of Equation~(\ref{E9.a}),
(\ref{E9.b}), and (\ref{E9.c}). We have:
\begin{eqnarray*}
&&R(\xi_1,\xi_2,\xi_2,\xi_1)=a_{11}^2f_{yy}=\lambda^2,\\
&&\nabla R(\xi_1,\xi_2,\xi_2,\xi_1;\xi_2)=a_{11}^2f_{yyy}=\lambda^3,\\
&&\nabla^2R(\xi_1,\xi_2,\xi_2,\xi_1;\xi_2,\xi_2)=a_{11}^2(x)f_{yyyy}=\lambda^4{ c_{122122}},\\
&&\frac{f_{yy}\cdot f_{yyyy}}{f_{yyy}\cdot f_{yyy}}=
\frac{\lambda^2a_{11}^{-2}\cdot\lambda^4c_{11}a_{11}^{-2}}{\lambda^6a_{11}^{-4}}={ c_{122122}}\,.
\end{eqnarray*}
Thus ${ c_{122122}}$ is an invariant of the theory; this will imply the 3 families of the theory fall into
different local isometry types. The ordinary differential equation
$\frac{\alpha\alpha^{\prime\prime}}{\alpha^\prime\alpha^\prime}={ c_{122122}}$
has the solutions $\alpha>0$ (see, for example, Lemma 1.5.5 of  \cite{G07}) of the form:
$$\alpha(t)=\left\{\begin{array}{ll}
e^{a(t+b)}&\text{ if }{ c_{122122}}=1\\
a(t+b)^c&\text{ if }{ c_{122122}}\ne1\end{array}\right\}
\text{ for }a\ne0\text{ and }c\ne0\,.$$
Thus
\begin{equation}\label{E9.d}
f_{yy}=\left\{\begin{array}{ll}
e^{\alpha(x)(y+\beta(x))}&\text{ if }{ c_{122122}}=1\text{ for }\alpha(x)\ne0\\
\alpha(x)(y+\beta(x))^c&\text{ if }
{ c_{122122}}\ne1\text{ for }\alpha(x)\ne0\text{ and }c\ne0
\end{array}\right\}\,.\end{equation}

We wish to simplify Equation~(\ref{E9.d}) to take $\beta(x)=0$.
We consider the change of variables
$T(x,y,z)=(x,y-\beta(x),\tilde x+y\beta_x(x))$:
$$\begin{array}{lrrr}
T_*\partial_x=(&1,&-\beta_x(x),&y\beta_{xx}(x)),\\
T_*\partial_y=(&0,&1,&\beta_x(x)),\\
T_*\partial_{\tilde x}=(&0,&0,&1).
\end{array}$$
We compute:
\begin{eqnarray*}
&&g_f(T_*\partial_x,T_*\partial_x)
=-2f(x,y-\beta(x))+\beta_x^2(x)+2y\beta_{xx}(x),\\
&&g_f(T_*\partial_x,T_*\partial_y)=0,
\qquad g_f(T_*\partial_x,T_*\partial_{\tilde x})=1,\\
&&g_f(T_*\partial_y,T_*\partial_y)=1,\qquad
g_f(T_*\partial_y,T_*\partial_{\tilde x})=g_f(T_*\partial_{\tilde x},T_*\partial_{\tilde x})=0.
\end{eqnarray*}
Thus $T^*g_f=g_{\tilde f}$ where
$$\tilde f(x,y)=f(x,y-\beta(x))-\textstyle\frac12
\left\{\beta_x^2(x)+2y\beta_{xx}(x)\right\}\,.$$
Consequently,
$\tilde f_{yy}(x,y)=f_{yy}(x,y-\beta(x))$.
Thus we may assume henceforth that $\beta(x)=0$ in
Equation~(\ref{E9.d}), i.e.
$$
f_{yy}=\left\{\begin{array}{ll}
e^{\alpha(x)y}&\text{ if }c_{11}=1\text{ for }\alpha(x)\ne0\\
\alpha(x)y^c&\text{ if }c_{11}\ne1\text{ for }\alpha(x)\ne0\text{ and }c\ne0
\end{array}\right\}$$
We examine these two cases seriatim. We shall use the relations:
\begin{eqnarray}
\lambda&=&f_{yyy}f_{yy}^{-1},\quad
a_{12}=-f_{xyy}f_{yyy}^{-1},\quad
\lambda^2a_{11}^{-2}=f_{yy},\label{E9.e}\\
\lambda^4{ c_{122112}}&=&\nabla^2R(\xi_1,\xi_2,\xi_2,\xi_1;\xi_1,\xi_2)=
a_{11}^3\{f_{xyyy}+a_{12}f_{yyyy}\},\label{E9.f}\\
\lambda^4{ c_{122111}}&=&\nabla^2R(\xi_1,\xi_2,\xi_2,\xi_1;\xi_1,\xi_1)\label{E9.g}\\
&=&a_{11}^4\{f_{xxyy}+2a_{12}f_{xyyy}+a_{12}^2f_{yyyy}-f_yf_{yyy}\}\,.
\nonumber\end{eqnarray}
\medbreak\noindent{\bf Case I.} Suppose $f_{yy}=e^{\alpha(x)y}$.
Then Equation~(\ref{E9.e}) implies:
$$
\lambda=f_{yyy}f_{yy}^{-1}=\alpha(x),\
 a_{12}=-f_{xyy}f_{yyy}^{-1}=-{ y}\alpha_x(x)\alpha(x)^{-1},\
\lambda^2a_{11}^{-2}=f_{yy}=e^{\alpha(x)y}\,.
$$
We use Equation~(\ref{E9.f})
to see that:
\begin{eqnarray*}
&&f_{xyyy}+a_{12}f_{yyyy}=\partial_x\{\alpha(x)e^{\alpha(x)y}\}
-\alpha_x(x)\alpha(x)e^{\alpha(x)y}\\
&=&
\alpha_x(x)\cdot e^{\alpha(x)y}=a_{11}^{-3}\lambda^4c_{12}=\alpha(x)e^{\frac32\alpha(x)y}
{ c_{122112}}\,.
\end{eqnarray*}
It now follows that $\alpha_x(x)=0$ so $\alpha(x)=a$ is constant. Thus
we may express:
$$f(x,y)=a^{-2}e^{ay}+u(x)y+v(x)\,.$$
We then use Equation~(\ref{E9.e}) to see
$$\lambda=a,\quad a_{12}=0,\quad \lambda^2a_{11}^{-2}=e^{ay}\,.$$
Equation~(\ref{E9.g}) then leads to the identity:
$$e^{2ay}{ c_{122111}}=a_{11}^{-4}\lambda^4{ c_{122111}}=-f_yf_{yyy}=-e^{2ay}-u(x)ae^{ay}\,.$$
This implies that $u(x)=0$ and hence $f=a^{-2}e^{ay}+v(x)$. Let $w_x=v(x)$ and set:
$$\begin{array}{rrrr}
T(x,y,\tilde x)=(&x,&y,&\tilde x+2w(x)),\\
T_*\partial_x=(&1,&0,&2v(x)),\\
T_*\partial_y=(&0,&1,&0),\\
T_*\partial_{\tilde x}=(&0,&0,&1).
\end{array}$$
Let $\Theta:=(x,y,\tilde x)$ and let $\tilde\Theta:=T\Theta$.
 Under this change of variables:
\begin{eqnarray*}
&&g(T_*\partial_x,T_*\partial_x)(\tilde\Theta)
=-2a^{-2}e^{ay}-2v(x)+2v(x)=-2a^{-2}e^{ay},\\
&&g(T_*\partial_x,T_*\partial_y)(\tilde\Theta)
=g(T_*\partial_y,T_*\partial_{\tilde x})(\tilde\Theta)
=g(T_*\partial_{\tilde x},T_*\partial_{\tilde x})(\tilde\Theta)=0,\\
&&g(T_*\partial_x,T_*\partial_{\tilde x})(\tilde\Theta)
=g(T_*\partial_y,T_*\partial_y)(\tilde\Theta)=1\,.
\end{eqnarray*}
Thus we may take $f=a^{-2}e^{ay}$. Replacing $y$ by $y+y_0$ for
suitably chosen $y_0$, then replaces $f$ by $e^{ay}$ as desired.

We now show $\mathcal{M}_{e^{ay}}$ is a homogeneous space. 
Clear the previous notation and set:
$$T(x,y,\tilde x)=(\pm e^{-ay_0/2}x+x_0,y+y_0,\pm e^{ay_0/2}\tilde x
+\tilde x_0)\,.$$
Then
$$T_*\partial_x=\pm e^{-ay_0/2}\partial_x,\quad
T_*\partial_y=\partial_y,\quad
T_*\partial_{\tilde x}=\mp e^{ay_0/2}\partial_{\tilde x}\,.
$$
Set $\Theta:=(x,y,\tilde x)$ and $\tilde\Theta:=T\Theta$. We show
that $T$ is an isometry by verifying:
\begin{eqnarray*}
&&g(T_*\partial_x,T_*\partial_x)(\tilde\Theta)=-2e^{-ay_0}e^{a(y+y_0)}
=g(\partial_x,\partial_x)(P),\\
&&g(T_*\partial_x,T_*\partial_y)(\tilde\Theta)=0,\quad
g(T_*\partial_x,T_*\partial_{\tilde x})(\tilde\Theta)=1,\quad
g(T_*\partial_y,T_*\partial_y)(\tilde\Theta)=1,\\
&&g(T_*\partial_y,T_*\partial_{\tilde x})(\tilde\Theta)=0,\quad
g(T_*\partial_{\tilde x},T_*\partial_{\tilde x})(\tilde\Theta)=0.
\end{eqnarray*}
Since $(x_0,y_0,\tilde x_0)$ are arbitrary,
$\mathcal{I}(\mathcal{M}_{e^{ax}})$
acts transitively on $\mathbb{R}^3$ so this manifold is globally homogeneous.
This verifies Theorem~\ref{T1.14}~(3a).

\medbreak\noindent{\bf Case II.}
Suppose that $f_{yy}=\alpha(x)y^c$ for $\alpha(x)>0$ and $c\ne0$.
Equation~(\ref{E9.e}) yields
$$
\lambda=f_{yyy}f_{yy}^{-1}=cy^{-1},\quad
a_{12}=-f_{xyy}f_{yyy}^{-1}=-\textstyle\frac{a_x(x)y}{c\alpha(x)},
\quad\lambda^2a_{11}^{-2}=f_{yy}=\alpha(x)y^c\,.
$$
We apply Equation~(\ref{E9.f}) to see:
\begin{eqnarray*}
&&\textstyle f_{xyyy}+a_{12}f_{yyyy}=\alpha_x(x)cy^{c-1}
-\frac{\alpha_x(x)y}{c\alpha(x)}c(c-1)\alpha(x)y^{c-2}
=\alpha_x(x)y^{c-1}\\
&=&a_{11}^{-3}\lambda^4{ c_{122112}}=\alpha(x)^{3/2}y^{3/2c}cy^{-1}{ c_{122112}}\,.
\end{eqnarray*}
This implies
$$\alpha_x(x)\alpha(x)^{-3/2}=c\cdot { c_{122112}}\cdot y^{c/2}\,.$$
Consequently $\alpha_x(x)=0$ so $\alpha(x)=a$ is constant. Consequently,
$f_{yy}=ay^c$ for $c\ne0$ and $a\ne0$.
Let $P(t)$ solve the equation $P^{\prime\prime}(t)=t^c$. We then have
$$f({ x,}y)=aP(y)+u(x)y+v(x)\,.$$
We apply Equation~(\ref{E9.g}) with $a_{12}=0$:
\begin{eqnarray*}
&&a_{11}^{-4}\nabla^2R(\xi_1,\xi_2,\xi_2,\xi_1;\xi_1,\xi_1)=
-f_yf_{yyy}=-aP^\prime(y)acy^{c-1}-u(x)acy^{c-1}\\
&=&{ c_{122111}}\lambda^4a_{11}^{-4}={ c_{122111}}a^2y^{2c}\,.
\end{eqnarray*}
If $c=-1$, then $P^\prime(y)=\ln(y)$ and this relation is impossible. Consequently $c\ne-1$ and
we may conclude that $u(x)=0$.
We therefore have $f=aP(y)+v(x)$. As in Case I, the constant term is
eliminated and $a$ is set to $1$ by making a change
of variables
$$T(x,y,\tilde x)=(a^{-1/2}x,y,a^{1/2}\tilde x+2w(x))$$
where $w_x(x)=v(x)$. Thus $f=\pm\ln(y)$ or $f=\pm y^\varepsilon$ for
$\varepsilon\ne0,1,2$.

\medbreak\noindent{\bf Case II-a.} Let $f(y)=\ln(y)$; the case
$f(y)=-\ln(y)$ is similar. We know by Theorem~\ref{T1.13}
that $\mathcal{M}_f$ is not $2$-curvature homogeneous and hence
is not homogeneous. 
We clear the previous notation. For $\lambda>0$ and $(x_0,\tilde x_0)$ arbitrary, set:
$$T(x,y,\tilde x):=(\lambda x+x_0,\lambda y,\lambda\tilde x
+\tilde x_0+{ (}\lambda\ln\lambda{ )} x)\,.$$
Let $\Theta:==(x,y,\tilde x)$ and $\tilde\Theta=T\Theta$. We compute:
\begin{eqnarray*}
&&T_*\partial_x=\lambda\partial_x+\lambda\ln\lambda\partial_{\tilde x},
\quad T_*\partial_y=\lambda\partial_y,\quad
T_*\partial_{\tilde x}=\lambda\partial_{\tilde x},\\
&&g(T_*\partial_x,T_*\partial_x)(\tilde\Theta)
=\lambda^2\{-2\ln(y)-2\ln\lambda\}+2\lambda^2\ln\lambda
=\lambda^2g(\partial_x,\partial_x)(P),\\
&&g(T_*\partial_x,T_*\partial_y)(\tilde\Theta)=0,\quad
g(T_*\partial_x,T_*\partial_{\tilde x})(\tilde\Theta)=\lambda^2,\quad
g(T_*\partial_y,T_*\partial_y)(\tilde\Theta)=\lambda^2,\\
&&g(T_*\partial_y,T_*\partial_{\tilde x})(\tilde\Theta)=0,\quad
g(T_*\partial_{\tilde x},T_*\partial_{\tilde x})=0\,.
\end{eqnarray*}
This defines a transitive action by homotheties on
$\mathbb{R}\times\mathbb{R}^+\times\mathbb{R}$, 
which shows that $\mathcal{M}_{\ln(y)}$ 
is homothety homogeneous. Moreover $T$ acts by isometries on
each level set of the projection
$(x,y,\tilde x)\mapsto y$, showing that $\mathcal{M}_{\ln(y)}$ is of cohomogeneity one.

\medbreak\noindent{\bf Case II-b.} Let $f(y)=y^c$ for $c\ne0,1,2$.
Again, Theorem~\ref{T1.13} implies $\mathcal{M}_f$ is not
$2$-curvature homogeneous and hence not homogeneous. Let
$$T(x,y,\tilde x)=(\lambda^{(2-c)/2}x+x_0,\lambda y,
\mp\lambda^{2+(c-2)/2}\tilde x+\tilde x_0)\,.$$
We compute:
\begin{eqnarray*}
&&T_*\partial_x=\lambda^{(2-c)/2}\partial_x,\quad
T_*\partial_y=\lambda\partial_y,\quad
T_*\partial_{\tilde x}=\lambda^{2+(c-2)/2}\partial_{\tilde x},\\
&&g(T_*\partial_x,T_*\partial_x)({ T(x,y,\tilde x)})=
-2\lambda^{(2-c)}\lambda^cy^c=\lambda^2g(\partial_x,\partial_x)({ x,y,\tilde x}),\\
&&g(T_*\partial_x,T_*\partial_y)=0,\quad
g(T_*\partial_x,T_*\partial_{\tilde x})=\lambda^2,\quad
g(T_*\partial_y,T_*\partial_y)=\lambda^2,\\
&&g(T_*\partial_y,T_*\partial_{\tilde x})=0,\quad
g(T_*\partial_{\tilde x},T_*\partial_{\tilde x)}=0\,.
\end{eqnarray*}
Thus $T$ is a homothety; since $\lambda>0$ is arbitrary and since
$(x_0,\tilde x_0)$ are arbitrary, the group of homotheties acts transitively
on $M$. This verifies Theorem~\ref{T1.14}~(3c) and completes
the proof of Theorem~\ref{T1.14}.\hfill\qed

\section{The proof of Theorem~\ref{T1.16}}\label{sect-7}

We consider the family of Lorentzian manifolds
given in Section~\ref{S9} and consider the family of models where
$\langle\xi_1,\xi_3\rangle=\langle\xi_2,\xi_2\rangle=1$
defines the inner product on the vector space
$V=\operatorname{Span}\{\xi_1,\xi_2,\xi_3\}$.
Choose the isometry between $T_PM$ and $(V,\langle\,\cdot\,,\,\cdot\,\rangle)$
 to be defined by:
$$\xi_1=a_{11}(\partial_x+f\partial_{\tilde x}+a_{12}\partial_y
+a_{13}\partial_{\tilde x}),\quad
\xi_2=\partial_y+a_{23}\partial_{\tilde x},\quad
\xi_3=a_{11}^{-1}\partial_{\tilde x}\,.
$$
To ensure the map in question is an isometry, we require
$a_{12}^2+2a_{13}=0$ and $a_{12}+a_{23}=0$.
This normalizes the parameters $a_{13}$ and $a_{23}$; $a_{11}$ and $a_{12}$ and $\lambda$ are the free parameters
of the theory.

We take $f(x,y)=\frac12\alpha(x)y^2$. The only (possibly)
non-zero covariant derivatives are given by:
$$
\nabla^\ell R(\partial_x,\partial_y,\partial_y,\partial_x;\partial_x,\dots,\partial_x)=\alpha^{(\ell)}(x)\,.
$$
The parameter $a_{12}$ plays no role in the theory and we have
$$\nabla^\ell R(\xi_1,\xi_2,\xi_2,\xi_1;\xi_1,\dots,\xi_1)=a_{11}^{2+\ell}\alpha^{(\ell)}\,.$$

Suppose $k$ is given. Let $\alpha^{(k)}:=e^{-x^2}$ and recursively define
\begin{equation}\label{E10.a}
\alpha^{(\ell)}(x):=\int_{t=-\infty}^x\alpha^{(\ell+1)}(t)dt\text{ for }0\le\ell\le k-1\,.
\end{equation}
For $x\le -1$, $-x^2\le x$ and consequently $\alpha^{(k)}\le e^{x}$.
Inductively, suppose $\alpha^{(\ell)}\le e^{x}$
for $x\le-1$. Then integration yields $\alpha^{(\ell-1)}\le e^x$ for $x\le-1$ as well. Thus the
integrals in Equation~(\ref{E10.a}) converge and $\alpha$ is a smooth function.
We have $\alpha^{(\ell)}>0$ for $0\le\ell\le k$; setting the functions $a_{11}^{\ell}=\{\alpha^{(\ell)}\}^{{ -}1/(2+\ell)}$ then yields
$$\nabla^\ell R(\xi_1,\xi_2,\xi_2,\xi_1;\xi_1,\dots,\xi_1)=1$$
and shows $\mathcal{M}_{\frac{1}{2}\alpha(x)y^2}$ is variable $k$-curvature homogeneous. Since, however,
$\alpha^{(k+1)}$ vanishes when $x=0$ and is non-zero when $x\ne0$,
$\mathcal{M}_{\frac{1}{2}\alpha(x)y^2}$ is
not variable $k+1$-curvature homogeneous on any neighborhood
of $0$. This establishes
Theorem~\ref{T1.16}~(1).

We take $\alpha(x)=e^x$. Then $\alpha^{(k)}(x)>0$ for all $x$ and for all $k$. Thus
the argument given above shows $\mathcal{M}_{\frac{1}{2}e^xy^2}$ is variable $k$-curvature homogeneous
for all $k$. On the other hand, by Theorem~\ref{T1.14}~(2),
$\mathcal{M}_{\frac{1}{2}e^xy^2}$ is not homothety 1-curvature homogeneous.
This establishes Theorem~\ref{T1.16}~(2).

We consider cases to establish Theorem~\ref{T1.16}~(3):
\begin{enumerate}
\item We use Theorem~\ref{T1.13}~(1b)
and Theorem~\ref{T1.14}~(1) to see that
the implication $(1b)\Rightarrow(1a)$ fails. This also shows the implication $(2b)\Rightarrow(1a)$ fails.
\item If we take $f=\frac12e^xy^2$, then
this is variable $k$-curvature homogeneous for all $k$ by Theorem~\ref{T1.16}~(2)
but not, by Theorem~\ref{T1.14}~(2) homothety 1-curvature homogeneous.
 Consequently the implications
$(2a)\Rightarrow(1a)$, $(2a)\Rightarrow(1b)$, $(2b)\Rightarrow(1a)$, and $(2b)\Rightarrow(1b)$
fail.
\item Let $(S^2,g_2)$ denote the standard round sphere in $\mathbb{R}^3$. Let
$$\mathcal{M}_t:=(M,g_t):=(\mathbb{R}\times S^2,e^{tx}(dx^2+g_2))$$
where $t\in(0,\epsilon)$ is a small positive real
parameter. We showed previously that this is homothety homogeneous
and hence satisfies (1b) for all $k$. We also showed it was not
$0$-curvature homogeneous for generic values of $t$. Note that
variable $0$-curvature and $0$-curvature homogeneous are the same.
Thus the implications $(1b)\Rightarrow(2a)$ and $(2b)\Rightarrow(2a)$ fail.
\end{enumerate}

\section*{Acknowledgments}
Research of all the authors was partially supported by programs with FEDER funds of Xunta de Galicia (Grant GRC2013-045). Research of P. Gilkey and S. Nik\v cevi\'c was also
partially supported by project 174012 (Serbia).

\end{document}